\documentclass[12pt]{amsart}
\usepackage{categorytheory}
\usepackage{ mathrsfs }
\usepackage{diags}
\fullpage

\tikzset{
  shiftup/.style = {transform canvas={yshift=0.3ex}},
  shiftdown/.style = {transform canvas={yshift=-0.3ex}},
}

\newcommand{\edit}[1]{}

\DeclareMathOperator{\Gal}{Gal}

\DeclareMathOperator{\Br}{Br}
\DeclareMathOperator{\Spec}{Spec}
\DeclareMathOperator{\rmfrac}{frac}
\DeclareMathOperator{\ann}{ann}

\def\map#1{\rto^{#1}}

\newcommand{\F}{\mathcal{F}}
\newcommand{\X}{\mathfrak{X}}
\newcommand{\frakm}{\mathfrak{m}}
\newcommand{\frakp}{\mathfrak{p}}
\renewcommand{\H}{\mathscr{H}}
\renewcommand{\a}{\underline{a}}
\newcommand{\mathdot}{{\mathbf{\scriptscriptstyle\bullet}}}
\newcommand{\et}{{\mathrm{et}}}
\newcommand{\Sm}{\mathbf{Sm}}

\newcommand{\LL}{\mathbb{L}}
\newcommand{\HH}{\mathbb{H}}
\newcommand{\tr}{\mathrm{tr}}
\newcommand{\Rtr}{R_{\tr}}
\newcommand{\Ztr}{\Z_{\tr}}
\newcommand{\grX}{\mathfrak{X}}
\newcommand{\DM}{\mathbf{DM}}
\newcommand{\cf}{{\it cf.\ }}
\def\oo{\otimes}
\def\muell#1{\mu_\ell^{\oo #1}}
\def\nis{{\text{nis}}}

\numberwithin{equation}{section}
\newtheorem{thm}[equation]{Theorem}
\newtheorem{prop}[equation]{Proposition}
\newtheorem{lem}[equation]{Lemma}
\newtheorem{cor}[equation]{Corollary}
\newtheorem{question}[equation]{Question}

\theoremstyle{definition}
\newtheorem{defn}[equation]{Definition}
\theoremstyle{remark}
\newtheorem{rem}[equation]{Remark}
\newtheorem*{remm}{Remark}
\newtheorem{exam}[equation]{Example}
\newtheorem{exams}[equation]{Examples}

\title{Principal ideals in mod-$\ell$ Milnor $K$-theory}
\author{Charles Weibel}
\address{Math.\ Dept., Rutgers University, New Brunswick, NJ 08901, USA}
\email{weibel@math.rutgers.edu}\urladdr{http://math.rutgers.edu/~weibel}
\thanks{Weibel was supported by NSA and NSF grants, 
and by the IAS Fund for Math}
\thanks{Zakharevich was supported by an NSF MSRFP, the University of
  Chicago and by the IAS Fund for Math}
\author{Inna Zakharevich}
\address{Math.\ Dept., University of Chicago, Chicago, IL 60637, USA}
\email{zakh@math.uchicago.edu}\urladdr{http://math.uchicago.edu/~zakh}
\date{\today}

\begin{document}

\begin{abstract}
Fix a symbol $\underline{a}$ in the mod-$\ell$ Milnor $K$-theory of a field
$k$, and a norm variety $X$ for $\underline{a}$. We show that the ideal
generated by $\underline{a}$ is the kernel of the $K$-theory map induced by
$k\subset k(X)$ and give generators for the annihilator of the ideal.  When
$\ell=2$, this was done by Orlov, Vishik and Voevodsky.
\end{abstract}

\maketitle


Let $\ell$ be a prime and $k$ a field containing $1/\ell$.  Given units
$a_1,\ldots,a_n\in k^\times$ we can form the Steinberg symbol 
$\a = \{a_1,\ldots,a_n\}$ in $K^M_n(k)$; we wish to study the ideal
$(\a)$  generated
by $\a$ in $K^M_n(k)/\ell$.  What is the quotient ring $(K^M_*(k)/\ell)/(\a)$,
and what is the annihilator ideal $\ann(\a)$, so that $(\a) =
(K^M_*(k)/\ell)/\ann(\a)$?  

Here is the main result of this paper; it was proven for $\ell=2$ by
Orlov, Vishik and Voevodsky in \cite[2.1]{ovv}.

\begin{thm}\label{ovv3.3}
Suppose that $\mathrm{char}\,k = 0$, and let $X$ be a 
norm variety for a nontrivial symbol
$\a$ in $K^M_n(k)/\ell$.  Then:
\begin{itemize}

\item[(a)] the kernel of $K^M_*(k)/\ell\map{}K^M_*(k(X))/\ell$
is the ideal of $K^M_*(k)/\ell$ generated by $\a$;

\item[(b)] the annihilator of $\a$ is the ideal of $K^M_*(k)/\ell$
generated by the norms
\[ 
\{ N(\alpha) \in K^M_*(k)/\ell \,|\, \alpha \in K^M_*(k(x)),\ 
x \textrm{ a closed point in }X \}.
\]
\end{itemize}
\end{thm}

Theorem \ref{ovv3.3} uses the notion of a {\it norm variety}; 
see Definition \ref{def:splitnorm} below.  
The existence of norm varieties is due to Rost; 
the terminology comes from \cite{SJ} and \cite[1.18]{HW}.  

\begin{exams}\label{ex:i=0}
Theorem~\ref{ovv3.3}(a) implies that $K^M_i(k)/\ell \rto K^M_i(k(X))/\ell$
is an injection when $i<n$, that
the kernel of $K^M_n(k)/\ell\rto K^M_n(k(X))/\ell$ is exactly the cyclic
subgroup generated by $\a$ and that the kernel of 
$K^M_{n+1}(k)/\ell\rto K^M_{n+1}(k(X))/\ell$ is the subgroup
$\a\cup k^\times$.

The group of units $b$ in $k^\times/k^{\times\ell}$
such that $\{ a_1,\dots,a_n,b\}=0$ in $K^M_{n+1}(k)/\ell$ forms 
the degree~1 part of the ideal $\ann(\a)$.  This group,
described in Theorem \ref{ovv3.3}(b) was originally described by Voevodsky.
If $H_{p,q}(X)$ is the motivic homology of a norm variety for $\a$, $X$, 
and $k$ has no extensions of degree $\ell$,
Voevodsky proved in \cite[A.1 and 2.9]{SJ} that the pushforward
$\pi_*:H_{-1,-1}(X)\rto H_{-1,-1}(\Spec k)=k^\times$ induces
an exact sequence
\addtocounter{equation}{-1}
\begin{subequations}
\renewcommand{\theequation}{\theparentequation\alph{equation}}
\begin{equation} \label{eq:motiv}
1\rto\bar H_{-1,-1}(X) \map{\pi_*} k^\times \map{\a\cup} K^M_{n+1}(k)/\ell.
\end{equation}
Here $\bar H_{p,q}(X)$ denotes the coequalizer of the two projections
  $H_{p,q}(X\times X)\rightrightarrows H_{p,q}(X)$.
Thus the degree~1 part of $\ann(\a)$ is $\bar H_{-1,-1}(X)$: 
$\{\a,b\}=0$ if and only if $b\in\bar H_{-1,-1}(X)$.

When $n=1$, write $\a=(a)$ for $a\in k^\times$, and set 
$E=k(\root\ell\of{a})$. Then $X=\Spec(E)$ is a norm variety for $\a$.
For simplicity, suppose that $k$ contains an $\ell$-th root of unity, $\zeta$.
The degree~2 part of $(\a)$ is the group of symbols $a\cup b$;
under the isomorphism $H_\et^2(k,\Z/\ell)\cong {}_\ell\!\Br(k)$,
$a\cup b$ is identified with the class of the cyclic algebra 
$A_\zeta(a,b)$ in the Brauer group. Theorem \ref{ovv3.3}
describes the units $b$ for which $A_\zeta(a,b)$ is 
a matrix algebra, and the division algebras (or classes
$[A]\in {}_\ell\!\Br(k)$) which are equivalent to cyclic algebras.
In this case, Kummer theory gives the answer: 
the first group is the image $N(E^\times)$ of the norm map 
$E^\times\rto k^{\times}$, 
and the second group is the class of algebras split by $E$.
(See \cite[6.4.8]{WK}.) 
In fact, we have the classical exact sequence
\begin{equation}\label{eq:n=1i=1}
1 \rto N(E^\times) \rto k^\times  
\map{a\cup} H_\et^2(k,\Z/\ell) \rto H_\et^2(E,\Z/\ell)^{\Gal(E/k)}.
\end{equation}

When $n=1$, Theorem \ref{ovv3.3} states that for every unit $a$ not in
$k^{\times\ell}$ there are exact sequences
\begin{equation}\label{eq:n=1}
1 \rto K^M_i(E)_{\Gal(E/k)} 
\rto^{N} K^M_i(k) \rto^{\cup a} K^M_{i+1}(k)/\ell
\rto (K^M_{i+1}(E)/\ell)^{\Gal(E/k)};
\end{equation}
\end{subequations}
when $i=1$ this is exactly (\ref{eq:n=1i=1}).  This follows from Voevodsky's
Galois computations \cite[3.2 and 3.6]{HW} (\cf \cite[5.2 and 6.11]{mc/2}) and
the fact that $\ell\cdot K^M_i(k)\subseteq N(K^M_i(E))$.
\end{exams}

\smallskip
Theorem \ref{ovv3.3} follows from the more technical Theorem \ref{thm:2->l}.
We note that the analysis in \cite{ovv} did not need to worry about 
roots of unity, as any field of characteristic $0$ contains the 
square roots of unity, and Pfister quadrics always have points
of degree~2.  For an odd prime $\ell$, the existence of a norm variety 
with points of degree $\ell$ is established in \cite[1.21]{SJ} modulo the
Norm Principle, proven in \cite[0.3]{HW-Rost}; see
Chapter 10 of \cite{HW}.

\begin{thm}\label{thm:2->l}
  Let $\mathrm{char}\,k = 0$.  Suppose that $X$ is a norm variety for a symbol
  $\a$ in $K^M_n(k)/\ell$ containing a point $x$ with $[k(x):k]=\ell$.  Write $q
  = n+i$ and let $\widetilde{K}^M_{q}(k(X))/\ell$ denote the equalizer
  of 
  maps $K^M_{q}(k(X))/\ell \rrto K_{q}^M(k(X\times X))/\ell$; $\X$ denotes the
  $0$-coskeleton of $X$.

\noindent(a)
If $\mu_\ell\subset k^\times$, there is an exact sequence for all $i$:
\[
\bar H_{-i,-i}(X) \map{\pi_*} K^M_i(k) \map{\a\cup} K^M_{q}(k)/\ell
\map{\iota} \widetilde{K}^M_{q}(k(X))/\ell \rto H^{q+1,q-1}(\X,\Z/\ell).
\]

\noindent(b) 
If $\mu_\ell\not\subset k^\times$, set $e=[k(\zeta):k]$ and
$X'=X\times_{k_1}k(\zeta)$, where $k_1=k(\zeta)\cap k(X)$.
If $\X'$ denotes the $0$-coskeleton of $X'$ over $k(\zeta)$, then
for all $i$ there is an exact sequence:
  \[
\bar H_{-i,-i}(X)[e^{-1}]\! \map{\pi_*}\! K^M_i(k)[e^{-1}]\! \map{\a\cup}\! 
K^M_{q}(k)/\ell \!\map{\iota}\! \widetilde{K}^M_{q}(k(X))/\ell\!
  \map{}\! H^{q+1,q-1}(\X',\Z/\ell)^G.
  \]
The map $\iota$ is induced by the homomorphism $k \rto k(X)$,
and $G = \Gal(k'/k_1)$.  
\end{thm}


The sequences \eqref{eq:motiv}, \eqref{eq:n=1i=1} and \eqref{eq:n=1} begin
with an injection.  This is often, but not always, the case.

\begin{question}
In the situation of Theorem~\ref{thm:2->l}(a) 
with $\mu_\ell\subset k^\times$, 
when is $\pi_*$ an injection?
\end{question}

For $i=0$, the map $\pi_*$ is an injection: $\bar{H}_{0,0}(X)=\Z$, and 
its image in $K_0(k)=\Z$ is $\ell\Z$. (This observation goes back to
\cite[8.7.2]{MS}.) This calculation
shows that the mod-$\ell$ reduction 
$\bar{H}_{0,0}(X,\Z/\ell)\rto K^M_0(k)/\ell$ 
of $\pi_*$ is not always an injection.

The map $\pi_*$ is an injection for $i=1$ by equation (\ref{eq:motiv}),
and for $n=1$
by Lemma \ref{lem:H90} below.
However, if $k$ does not contain the $\ell^{th}$ roots of unity, 
$\pi_*$ need not be an injection even for $i=n=1$, as 
the classical Hilbert Theorem 90 can fail; 
see Example \ref{ex:notgalois} below. 

Theorem~\ref{ovv3.3}(b) could be strengthened to only look at norms of
elements in $K^M_1(k)=k^\times$ if we knew that 
the answer to the following question was affirmative: 

\begin{question}\label{q:BassTate}
If $E/F$ is a Galois extension of prime degree,
is $K^M_{n+1}(E)$ always generated by symbols
$\{a_1,...,a_n,b\}$ with $a_i\in F^\times$ and $b\in E^\times$?

It suffices to check the case $n=1$: is $K^M_2(E)$ is always 
generated by symbols $\{a,b\}$ with $a\in F^\times$ and $b\in E^\times$?
\end{question}

If $\ell=2$, $\ell=3$ or $k$ is $\ell$-special, this is the case;
$K^M_2 k(x)$ is generated by symbols $\{a,b\}$ with $a\in k^\times$ and
$b\in k(x)^\times$; see \cite[Lemma 2]{M82}, \cite[p.\,388]{BT}.
By Becher \cite[1.1]{Becher},
$K^M_n k(x)$ is also generated by symbols $\{\alpha,\beta\}$ with
$\alpha\in K^M_{n-m}(k)$, $\beta\in K^M_m(k)$ if $\ell<2^{m+1}$.

The restriction to prime degree is necessary in Question \ref{q:BassTate}.
Becher has pointed out in \cite[3.1]{Becher} that if $E=k(x,y)$
and $F=k(x^\ell,y^\ell)$ then $\{x,y\}$ cannot be written in this
form, as the tame symbol 
$\partial_y:K^M_2(E)\rto k(x)^\times/k^{\times\ell}$ shows.
In this case, $[E:F]=\ell^2$.

\begin{rem}
Although most of our results work over perfect fields of 
arbitrary characteristic, the assumption that $k$ has 
characteristic 0 is needed in two places.
\begin{itemize}
\item[1)]To prove that norm varieties exist for symbols of length $n$.  This
  would go through for any perfect field of positive characteristic (by
  induction on $n$) if we could prove that for symbols of length $n-1$ over $k$,
  a norm variety $Y$ exists which satisfies the {\it Norm Principle} (see
  \cite[0.3]{HW-Rost} or \cite[10.17]{HW}).  The inductive step is given in
  \cite[10.21]{HW}.

\item[2)] We also need characteristic~0 to show that the symmetric
characteristic class $s_d(X)$ of a norm variety is nonzero modulo $\ell^2$.
The proof in characteristic~0 is due to Rost (unpublished), and
given in Proposition 10.13 of \cite{HW},
and depends upon the Connor--Floyd theory of equivariant cobordisms on 
complex $G$-manifolds (as given by Theorem 8.16 in {\it loc.\,cit.})
It is possible that a proof in characteristic $p>0$ could be given
along the lines of \cite[5.2]{SJ}, if we assume resolution of singularities.
\end{itemize}
We will therefore state as many of our results in as much generality as
possible, only restricting to characteristic zero when absolutely necessary.
\end{rem}

\begin{rem}
After writing this paper, we discovered that many of our results
are in Yagita's paper \cite[Thm.\,10.3]{Yagita} and in the Merkurjev-Suslin
paper \cite[2.1]{MS2}.  The basic technique in these papers, and in ours, 
is the same: generalize the ideas in \cite{ovv}, 
using Rost's norm varieties for $\ell>2$.  Yagita's proof is
somewhat sketchy, as it predated a clear understanding of norm varieties.
Merkurjev and Suslin prove Theorem \ref{ovv3.3}(b), but their formulation
is different in the absence of roots of unity.
Since neither of these results directly addresses
the ring structure of $K^M_*(k)/\ell$, we feel that our exposition
should be added to the public record.
\end{rem}

\subsection*{Notation and conventions} We fix a prime $\ell$ and
an $\ell$-th root of unity $\zeta$.  
We write $H^{p,q}(Y,\Z/\ell)$ for $H^p_\nis(Y,\Z/\ell(q))$.

\section{Borel--Moore homology}

The first term in Theorem \ref{thm:2->l} uses the
motivic homology group $H_{-i,-i}(X)$ of a smooth projective variety $X$
(with coefficients in $\Z$).
However, it is more useful to think of it as the Borel--Moore
homology group $H^{BM}_{-i,-i}(X)$, which is covariant 
for proper maps between smooth varieties, and contravariant 
for finite flat maps; see \cite[p.\,185]{biv} or \cite[16.13]{mvw}.

When $X$ is smooth projective, we have $H_{-i,-i}(X)=H^{BM}_{-i,-i}(X)$,
and more generally $H_{p,q}(X,\Z)=H^{BM}_{p,q}(X,\Z)$, because
the natural map from $M(X)=\Ztr(X)$ to $M^c(X)$ in $\DM$ is an 
isomorphism for smooth projective $X$. (Recall that the motivic 
homology groups $H_{p,q}(X,\Z)$ of $X$ are defined to be 
$\Hom_{\DM}(\Z(q)[p],M(X))$, while the 
Borel--Moore homology groups $H^{BM}_{p,q}(X,\Z)$ are defined to be 
$\Hom_{\DM}(\Z(q)[p],M^c(X))$;
see \cite[p.\,185]{biv} or \cite[14.17, 16.20]{mvw}.)
%
We define $H^{BM}_{-i,-i}(X)$ to be $H^{BM}_{-i,-i}(X,\Z)$ if
$\mathrm{char}\,k = 0$, and $H^{BM}_{-i,-i}(X,\Z[1/p])$ if
$\mathrm{char}\,k = p>0$.

The case $i=1$ of the following result was proven in \cite{SJ}.

\begin{prop}\label{H-n-n}
Let $X$ be a smooth variety over a perfect field $k$.
If $i\geq 0$ we have an exact sequence
  \[ 
  \coprod\nolimits_{y}K^M_{i+1}(k(y)) \map{\textrm tame} \coprod\nolimits_{x}
  K^M_i(k(x)) \map{} H^{BM}_{-i,-i}(X) \rto 0.
  \] 
In addition, $H^{BM}_{-i,-i}(X)$ is isomorphic to $H^{2d+i,d+i}(X,\Z)$. More
explicitly, $H^{BM}_{-i,-i}(X)$ is the abelian group 
generated by symbols $[x,\alpha]$, where $x$
is a closed point of $X$ and $\alpha\in K^M_i(k(x))$,
modulo the relations \\
  (i) $[x,\alpha][x,\alpha']=[x,\alpha+\alpha']$ and
  \\
  (ii) the image of the tame symbol $K^M_{i+1}(k(y))\rto\bigoplus
  K^M_{i}(k(x))\rto H^{BM}_{-i,-i}(X)$ is zero for every codimension~1 point $y$
  of $X$.
\end{prop}

\begin{proof}
Let $A$ denote the abelian group with generators $[x,\alpha]$
and relations (i) and (ii), described in the Proposition,
and set $d=\dim(X)$. 
We first show that $A$ is isomorphic to $H^d(X,\H^{d+i})$, where 
$\H^q$ denotes the Zariski sheaf associated to the presheaf 
$H^{q,d+i}(-,\Z)$. 
For each $q$, $\H^{q}$ is a homotopy invariant Zariski sheaf, 
by \cite[24.1]{mvw}. As such, it has a canonical flasque
``Gersten'' resolution on each smooth $X$ (given in \cite[24.11]{mvw}),
whose $c^{th}$ term is the coproduct over codimension~$c$ points $z$ of 
the skyscraper sheaves $H^{q-c,d+i-c}(k(z))$, where $z$ has codimension 
$c$ in $X$.  Taking $q=d+i$, and recalling that $K^M_i\cong H^{i,i}$ 
on fields, we see that the skyscraper sheaves in the $(d-1)^{st}$ and
$d^{th}$ terms take values in $K^M_{i+1}(k(y))$ and $K^M_i(k(x))$.
Moreover, 
the map $K^M_{i+1}(k(y))\rto K^M_i(k(x))$ is the tame symbol 
if $x\in\overline{\{ y\}}$, and zero otherwise. 
As $H^d(X,\H^{d+i})$ is obtained by
taking global sections and then cohomology, 
it is isomorphic to $A$.  

Next, we show that $A$ is isomorphic to $H^{2d+i,d+i}(X,\Z)$.
To this end, consider the hypercohomology spectral sequence 
$E_2^{p,q}=H^{p}(X,\H^q) \Rightarrow H^{p+q,d+i}(X,\Z),$
Since $H^{q,d+i}=0$ for $q>d+i$, 
the spectral sequence is zero unless
$p\le d$ and $q\le d+i$. From this we deduce that
$H^{2d+i,d+i}(X,\Z) \cong H^d(X,\H^{d+i})\cong A$. 

Finally, we show that $H^{BM}_{-i,-i}(X)$ is isomorphic to
$H^{2d+i,d+i}(X,\Z)$.
Suppose first that $i=0$. Then the presentation describes
$CH_0(X)\cong H^{2d,d}(X,\Z)$, and by \cite{V-CH} 
we also have $H_{0,0}^{BM}(X)=CH_0(X)$. Thus we may assume that $i>0$.


If $\textrm{char}(k)=0$, the proof is finished by the duality 
calculation, which uses Motivic Duality with $d=\dim(X)$ 
(see \cite[16.24]{mvw} or \cite[7.1]{biv}):
\[
\begin{aligned}
H^{BM}_{-i,-i}(X,\Z) =&
\Hom(\Z,M^c(X)(i)[i])=\Hom(\Z(d)[2d],M^c(X)(d+i)[2d+i]) \\
 =& \Hom_{\mathbf{DM}}(M(X),\Z(d+i)[2d+i])=H^{2d+i,d+i}(X,\Z). 
\end{aligned}
\]

Now suppose that $k$ is a perfect field of $\textrm{char}(k)=p>0$.
As we show below in Lemma  \ref{K2:p-divisible}, 
$K^M_i(k(x))$ and $K_{i+1}^M(k(y)$) are uniquely p-divisible for $i\ge1$
(when $x$ is closed in $X$ and $\textrm{trdeg}_k k(y)=1$).
Thus $A$ must also be uniquely p-divisible.
Since $H^{2d+i,d+i}(X,\Z)\cong\!A$, 
the duality calculation above goes through with $\Z$ replaced by 
$\Z[1/p]$, using the characteristic $p$ version of Motivic Duality (see
\cite[5.5.14]{Shanekelly}) and we have
$H^{BM}_{-i,-i}(X,\Z[1/p])\cong\! H^{2d+i,d+i}(X,\Z[1/p])\cong\! H^{2d+i,d+i}(X,\Z)$.
\end{proof}

\begin{lem}[Izhboldin]\label{K2:p-divisible}
Let $E$ be a field  of transcendence degree~$t$ over a perfect field
$k$ of characteristic~$p$. Then $K^M_{m}(E)$ is uniquely $p$-divisible
for $m>t$.
\end{lem}

\begin{proof}
For any field $E$ of characteristic~$p$, the group $K^M_m(E)$ 
has no $p$-torsion by Izhboldin's Theorem (\cite[III.7.8]{WK}), 
and the $d\log$ map $K^M_m(E)/p\rto\Omega^m_E$
is an injection with image $\nu(m)$; see \cite[III.7.7.2]{WK}.
Since $k$ is perfect, $\Omega^1_k=0$ and $\Omega^1_E$ is
$t$--dimensional, so if $m>t$ then $\Omega^m_E=0$ 
and hence $K^M_m(E)/p=0$.
\end{proof}

\begin{exam}\label{ex:N=norm}
(i) $H_{-i,-i}(\Spec E)=K^M_i(E)$ for every field $E$ over $k$,
as is evident from the presentation in Lemma \ref{H-n-n}. 
\\
(ii) If $E$ is a finite extension of $k$, the proper pushforward 
from $K^M_i(E)=H_{-i,-i}(\Spec E)$ to $K^M_i(k)=H_{-i,-i}(\Spec k)$
is just the norm map $N_{E/k}$; see \cite[III.7.5.3]{WK}. \\ 
(iii) If $\pi:X\rto\Spec(k)$ is proper, and $x\in X$ is closed,
the restriction of the pushforward 
\[
\pi_*:H_{-i,-i}(X)\rto H_{-i,-i}(\Spec k)=K^M_i(k)
\]
to $K^M_i(k(x))$ sends  $[x,\alpha]$ to the norm $N_{k(x)/k}(\alpha)$.
This follows from (ii) by functoriality of $H_{-i,-i}$ for the composite 
$\Spec k(x)\rto X \rto \Spec k$, $x\in X$ closed. 
From the presentation in Lemma \ref{H-n-n}, the map $N_{X/k}$ is
completely determined by the formula $\pi_*[x,\alpha]=N_{k(x)/k}(\alpha)$.

In particular, the image of $\pi_*$ is the subgroup of
$K^M_i(k)$ generated by the norms $N_{k(x)/k}(\alpha)$ of
$\alpha\in k(x)^\times$ as $x$ ranges over the closed points of $X$.
\end{exam}

\begin{lem}\label{lem:H90}
Suppose that $\mu_\ell\subset k$ and $a\in k^\times$, and set
$E=k(\root\ell\of{a})$, $X=\Spec(E)$. Then 
$\bar H_{-i,-i}(X)\cong K^M_i(E)_{\Gal(E/k)}$, and
$\bar H_{-i,-i}(X)\rto K^M_i(k)$ is an injection.
\end{lem}

\begin{proof}
Note that $E/k$ is Galois with group $G$, so $X\times X\cong\prod_GX$
and $\bar H_{-i,-i}(X)\cong (K^M_iE)_G$ by Example \ref{ex:N=norm}(i).
In this case, $(K^M_iE)_G$ is a subgroup of $K^M_i(k)$ by
\eqref{eq:n=1}. 
\end{proof}

\begin{exam}\label{ex:notgalois}
If $E/k$ is not Galois, $\bar H_{-i,-i}(\Spec(E))\rto K^M_i(k)$ 
need not be an injection, even for $n=1$.  One way to think of this is to 
realize that the classical Hilbert 90 asserts exactness of 
$(E\oo E)^\times \rightrightarrows E^\times \rto k^\times$, 
and Hilbert 90 requires $E/k$ to be Galois.  A concrete example 
is given by $\ell=3$, $k=\Q$, and $E=\Q(\root3\of2)$. In this
case, $\Spec(E)\times\Spec(E)\cong \Spec(E\times F)$, where 
$F=E(\root3\of1)$, and the coequalizer $\bar H_{-1,-1}(\Spec(E))$ 
of $(E\times F)^\times\rightrightarrows E^\times$ does not inject 
into $\Q^\times$.  This shows that $\pi_*$ in 
Theorem~\ref{thm:2->l}(a) is not always an injection.
\end{exam}

\medskip
\section{Norm varieties} \label{sec:prelim}

Let $\a=(a_1,\dotsc,a_n)$ be a sequence of units in a field $k$
of characteristic not equal to $\ell$.

\begin{defn} \label{def:splitnorm}
A field $F$ over $k$ is said to be a {\it splitting field} for $\a$
if $\a$ vanishes in $K_n^M(F)/\ell$.
We say that a variety $X$ is a {\it splitting variety} for $\a$
if $k(X)$ is a splitting field for $\a$, i.e.,
if $\a$ vanishes in $K_n^M(k(X))/\ell$. 

Let $X$ be a splitting variety for $\a$. We say that $X$ is 
an {\it $\ell$-generic} splitting variety for $\a$ if 
any splitting field $F$ has a finite extension $E$ of degree
prime to $\ell$ with $X(E)\ne\emptyset$.

A {\it norm variety} for $\a$ is a smooth projective variety $X$ 
of dimension $d=\ell^{n-1}-1$ 
which is an $\ell$-generic splitting variety for $\a$.
When $\textrm{char}(k)=0$, a norm variety for $\a$ always exists
(see \cite[10.16]{HW}).
\end{defn}

For example, $E=k(\root\ell\of a_1)$ is a splitting field for
$\a=(a_1,...,a_n)$. Since a norm variety $X$ is
$\ell$-generic, there is a finite field extension $E'/E$
of degree prime to $\ell$ and an $E'$-point of $X$.
The following result, due to Rost, is proven in
Chapter 10 of \cite{HW}.

\begin{thm}
If $\a$ is a nonzero symbol over $k$ and $\textrm{char}(k)=0$, 
then there exists a norm variety $X$ for $\a$ having a closed point
$x$ with $[k(x):k]=\ell$.
\end{thm}

We will frequently use the following fact, proven in 
\cite[1.21]{SJ} (see \cite[10.13]{HW}):
if $k$ has characteristic~0 and $n\ge2$,
the symmetric characteristic class $s_d(X)$ of a norm variety $X$ 
is nonzero modulo $\ell^2$
(i.e., $X$ is a {\it $\nu_{n-1}$-variety}).

\begin{defn}
  Given a norm variety $X$, let $\grX$ denote its 0-coskeleton, i.e., the
  simplicial scheme $p\mapsto X^{p+1}$ with the projections $X^{p+1}\rto X^p$ as
  face maps and the diagonal inclusions as degeneracies.
\end{defn}

For simplicity, we write  $\LL$ for $\Z_{(\ell)}(1)[2]$
and $\Rtr(\X)$ for $\Z_{(\ell)\;\tr}(\X)$, and regard $X$ as a Chow motive.
Recall \cite[20.1]{mvw} that Chow motives form a full subcategory of $\DM$.


\begin{thm}\label{thm:M}
Let $X$ be a norm variety for $\a$ 
such that $s_d(X)$ is nonzero modulo $\ell^2$.
Then there is a Chow motive $M=(X,e)$ with coefficients $\Z_{(\ell)}$,
such that 
\begin{enumerate}
\item[(i)]\enspace
$M=(X,e)$ is a symmetric Chow motive, i.e., $(X,e)=(X,e^t)$;
\item[(ii)]\enspace 
The projection $X\rto \Z_{(\ell)}$ factors as $X\rto(X,e)\rto \Z_{(\ell)}$, 
i.e., is zero on $(X,1-e)$;

\item[(iii)]\enspace
There is a motive $D$ related to the structure map $y:M\rto\Rtr(\grX)$ and
its twisted dual $y^D$ by two distinguished triangles in $\DM$,
where $b=d/(\ell-1)$:
\addtocounter{equation}{-1}
\begin{subequations}
\begin{align}
D\otimes\LL^b\ \rto\ &M\,\map{y}\,\Rtr(\grX)\map{s} D\oo\LL^b[1]\ , 
\label{3.4.1}\\
\Rtr(\grX)\oo\LL^d \map{y^D} &M\map{u} D\map{r} \Rtr(\grX)\oo\LL^d[1].
\label{3.4.2}
\end{align}
\end{subequations}
\end{enumerate}
\end{thm}

\begin{proof}
This is proven carefully in \cite[Ch.\,5]{HW}; the construction
is due to Voevodsky \cite[pp.\,422--428]{mc/l} 
and appears in Section\,1 of \cite{W-Top}. Specifically,
$\a$ determines a motive $A$ by (5.1), Definition 5.5 and 5.13.1 
of \cite{HW}; by definition, $M=S^{\ell-1}(A)$ and $D=S^{\ell-2}(A)$. 
Part (i) follows from 5.19; part (ii) follows 
from 5.9; and part (iii) follows from 5.7 of {\it loc.\,cit.}
\end{proof}


Although many of our techniques require the field $k$ to 
contain the $\ell$-th roots of unity, we can sometimes remove 
this restriction using the following observation.  Given a
norm variety $X$ over a field $k$, let $k_1$ denote the largest
subfield of $k(\zeta)$ contained in $k(X)$. Then
$X$ is also a norm variety for $\a$ over $k_1$.

\begin{lem}\label{basechange}
Given a nonzero symbol $\a\in K^M_*(k)/\ell$, let $X$ be a norm
variety for $\a$ over $k$. Then every component $X'$ of $X_{k(\zeta)}$
is a norm variety for $\a$ over $k(\zeta)$.
\end{lem}

\begin{proof}
  Clearly, $X'$ is a splitting variety for $\a$ of the right dimension.  Given a
  splitting field $F$ of $\a$ over $k(\zeta)$, there is a prime-to-$\ell$
  extension $E$ of $F$ such that $k(\zeta)\subset E$ and such that there exists
  a map $\Spec\,E\rto X$ over $k$.  By basechange, there is a map
  $\Spec\,E\oo_kk(\zeta)\rto X_{k(\zeta)}$ over $k(\zeta)$.  As $k(\zeta)\subset
  E$, $E\oo_kk(\zeta)$ is a $\Gal(k(\zeta)/k)$-indexed product of copies
  of $E$. Since $\Gal(k(\zeta)/k)$ acts transitively on the components
  of $X_{k(\zeta)}$, each component $X'$ of $X_{k(\zeta)}$ has an $E$-point.
  Thus $X'$ is a norm variety over $k(\zeta)$.
\end{proof}

\begin{rem}\label{rem:basechange}
$X_{k(\zeta)}$ is a $\Gal(k_1/k)$-indexed coproduct of copies of 
$X'=X\times_{k_1}\Spec\,k(\zeta)$.
\end{rem}

\medskip
\section{Reducing to Theorem~\ref{thm:2->l} over fields containing $\ell$-th roots}

We are now ready to prove Theorem~\ref{ovv3.3} assuming
Theorem~\ref{thm:2->l}.  Fix a field $k$ of characteristic~0, 
a symbol $\a$ and a norm variety $X$ for $\a$.  
We first observe that the statement of Theorem~\ref{ovv3.3} 
is equivalent to the exactness of the sequence
\begin{equation}\label{eq:seq=thm}
H_{-i,-i}(X)/\ell \map{\pi_*} K^M_i(k)/\ell \map{\a\cup} 
K^M_{i+n}(k)/\ell \map{\iota} K^M_{i+n}(k(X))/\ell .
\end{equation}
As observed in Example \ref{ex:i=0}, Theorem \ref{ovv3.3}
for $n=1$ follows from \eqref{eq:n=1} when $\mu_\ell\subset k^\times$.  

\begin{prop}
Suppose that Theorem~\ref{thm:2->l} holds over $k$.
Then so does Theorem~\ref{ovv3.3}.
\end{prop}

\begin{proof}
As the equalizer $\tilde K^M_{i+n}(k(X))/\ell$ is a subgroup of
$K^M_{i+n}(k(X))/\ell$, 
Theorem \ref{thm:2->l} implies that there is an exact sequence
\[
H_{-i,-i}(X)[e^{-1}] \map{\pi_*} K^M_i(k)[e^{-1}] \map{\a\cup} 
K^M_{i+n}(k)/\ell \map{\iota} K^M_{i+n}(k(X))/\ell .
\]
(If $\mu\subset k^\times$ then $e=1$).
Exactness of \eqref{eq:seq=thm} is immediate.
\end{proof}

Thus we have reduced the proof of Theorem~\ref{ovv3.3} to
Theorem~\ref{thm:2->l}.  We will now show that proving 
Theorem~\ref{thm:2->l} over fields containing $\ell$-th roots of unity
suffices. 

\begin{prop}\label{prop:5=>2} 
  Suppose that Theorem~\ref{thm:2->l} holds for all fields of
  characteristic $0$ which contain $\ell$-th roots of unity.  Then
  Theorem~\ref{thm:2->l} holds for all fields of characteristic $0$.
\end{prop}
  
\begin{proof}
  Let $k$ be any field of characteristic $0$ not containing an
$\ell^{th}$ root of unity, $\zeta$.  Set $q = n+i$, $k'=k(\zeta)$,
$k_1=k'\cap k(X)$, $e=[k':k]$ and $G = \Gal(k'/k_1)$, as in the
statement of Theorem \ref{thm:2->l}(b). 
By Lemma \ref{basechange} and Remark \ref{rem:basechange},
the component $X'=X\times_{k_1}\Spec(k')$ of $X_{k'}$ 
is a norm variety for $\a$ over $k'$. The action of $G$ on $k'$
induces actions of $G$ on $X'$ and its 0-skeleton $\X'$, 
and induces the last map in Theorem \ref{thm:2->l}(b):
\[
\tilde K_{q}^M(k(X))/\ell \map{j} (\tilde K_{q}^M(k'(X'))/\ell)^G
\map{\partial} H^{q+1,q-1}(\X')^G.
\]
Since $e$ is prime to $\ell$, inverting $e$ in the exact sequence of 
Theorem \ref{thm:2->l} for $k'$ yields the exact sequence forming the
bottom row of the following diagram, in which the downward arrows are
base change maps and the upward arrows are the norm maps.
\begin{diagram-fixed}[1.9em]
{ H_{-i,-i}(X)[e^{-1}]  & K^M_i(k)[e^{-1}] & K^M_{q}(k)/\ell & 
   \tilde K_{q}^M(k(X))/\ell 
& H^{q+1,q-1}(\X')^G \\
  H_{-i,-i}(X')[e^{-1}] & K^M_i(k')[e^{-1}] & K^M_{q}(k')/\ell &
 \tilde K_{q}^M(k'(X'))/\ell  &
 H^{q+1,q-1}(\X') \\};
    \to{1-1}{1-2}^{\pi_*} \to{1-2}{1-3}^{\a\cup} \to{1-3}{1-4}^{\iota}
    \to{1-4}{1-5}^{\partial j}
    \to{2-1}{2-2}^{\pi'_*} \to{2-2}{2-3}^{\a\cup} \to{2-3}{2-4}^{\iota}
    \to{2-4}{2-5}^{\partial}
    \cofib{1-2}{2-2} 
    \cofib{1-3}{2-3} 
    \to{1-4}{2-4}^{j}
    \to{1-5}{2-5}
    \diagArrow{->,bend left}{2-1}{1-1}^N
    \diagArrow{->,bend left}{2-2}{1-2}^N
    \diagArrow{->,bend left}{2-3}{1-3}^N
\end{diagram-fixed}
As each $K$-group is covariantly functorial, the diagram with the
downward set of arrows commutes; the diagram with the upward set of
arrows commutes by naturality and the projection formula \cite[III.7.5.2]{WK}.  
The downward map $K_*^M(k) \rto K^M_*(k')$, followed by the norm map,
is multiplication by $e=[k':k]$. 
A diagram chase now shows that the top row of the diagram is exact.
\end{proof}
%


\begin{rem}
The map $j$ is also injective in the above diagram. To see this,
note that (by the projection formula)
the norm $K^M_q(k'(X'))/\ell\rto K^M_q(k(X))/\ell$ induces
a map $\tilde N$ from 
$\tilde K^M_q(k(X'))/\ell$ to $\tilde K^M_q(k(X))/\ell$, and the composition 
$\tilde{N}\,j$ 
is multiplication by $[k':k_1]$, not $e$. Note that $\tilde{N}$ does not 
commute with the norm $K^M_q(k')/\ell\rto K^M_q(k)/\ell$ unless $k=k_1$.
\end{rem}

\section{The exact sequence} \label{sec:exact}

In this section and the next, we assume that our field $k$ contains
an $\ell$-th root of unity, $\zeta$. As before, we fix
a symbol $\a$ and a norm variety $X$ for $\a$, 
writing $\X$ for the $0$-coskeleton of $X$.

Given a complex $\F^\mathdot$ of \'etale sheaves, 
let $\H^q=\H_\nis^q(\F^\mathdot)$ denote the Nisnevich sheaf 
associated to the presheaf $H_\et^q(-,\F^\mathdot)$. 
If $\F$ is a locally constant \'etale
sheaf (such as $\muell{i}$), $\H^q(\F)$ is a Nisnevich sheaf with
transfers, by \cite[6.11, 6.21 and 13.1]{mvw}.

\begin{lem}\label{H0(Hq)}
If $\F$ is a sheaf, $H^0(\X,\H^q)$ is the equalizer of 
$H^0(X,\H^q) \rrto H^0(X\times X,\H^q)$.
\end{lem}

\begin{proof}
This is the definition of $H^0$ on a simplicial scheme; see \cite[5.2.2]{D}.
Alternatively, it follows from the spectral sequence 
$E_1^{p,q}=H^q(X^{p+1},\F)\Rightarrow H^{p+q}(\X,\F)$
for the cohomology of a sheaf on a simplicial scheme.
\end{proof}

\begin{rem}\label{rem:11.1}
The Nisnevich sheaves $\H^q(\muell{q})$ are homotopy invariant 
sheaves with transfers, by \cite[24.1]{mvw}. By  \cite[11.1]{mvw},
if $X$ is smooth then $H^0(X,\H^q(\muell{q}))$ --- and hence
$H^0(\X,\H^q(\muell{q}))$ --- injects into 
$\H^q(\muell{q})(\Spec k(X))=H_\et^q(k(X),\muell{q})\cong K^M_q(k(X))/\ell$.
\end{rem}

\begin{prop}\label{prop:coeff-triangle}
If $\mu_\ell\subset k^\times$,
there is a distinguished triangle in $\DM$ for each $q\ge0$:
  \[
\Z/\ell(q-1) \map{\zeta} \Z/\ell(q) \rto \H^q(\muell{q})[-q] \rto. 
\]
\end{prop}

\begin{proof}
For any Nisnevich complex $C$ and any $q$ we have a distinguished triangle 
\[
\tau^{\leq q-1}C  \rto \tau^{\leq q} C \rto H^q(C)[-q] \rto.
\]
Now let $C$ be the total direct image $R\pi_*\muell{q}$, where
$\pi:\Sm_\et\rto\Sm_\nis$, so $H_\nis^*(X,C)=H_\et^*(X,\muell{q})$.  
Since $\mu_\ell\subset k^\times$, multiplication by $\zeta$ induces 
an isomorphism $\muell{q-1}\cong\muell{q}$.
Thus we have an isomorphism $\cup\zeta:
R\pi_*\muell{q-1} \map{\simeq} C$.  In this case, the triangle reads:
\[
\tau^{\le q-1}R\pi_*(\muell{q-1})  \map{\zeta} 
\tau^{\le q}  R\pi_*(\muell{q}) \rto \H^q(\muell{q})[-q] \rto.
\]
By the Beilinson-Lichtenbaum conjecture (which has now been proven; 
see \cite[6.17]{mc/l} or \cite[Thm.\,B]{HW}), $\Z/\ell(q)\cong \tau^{\le q}C$ and
$\Z/\ell(q-1)\cong\tau^{\le q-1}R\pi_*\muell{q-1}\cong\tau^{\le q-1}C$.
Combining these facts yields the distinguished triangle in question.
\end{proof}


Let $\tilde\X$ denote the simplicial cone of $\X \rto \Spec k$.
As a consequence of the Beilinson-Lichtenbaum conjectures, 
Voevodsky observed that

\begin{lem}\label{ovv2.2}
If $X$ is smooth, the map $H^{p,q}(k,\Z/\ell)\map{} H^{p,q}(\X,\Z/\ell)$
is an isomorphism if $p\le q$ and an injection if $p=q+1$.
That is, $H^{p,q}(\tilde\X,\Z/\ell)=0$ if $p\le q+1$.
\end{lem}

\begin{proof}
See \cite[6.9 and 7.3]{mc/2} or \cite[1.37]{HW}.
\end{proof}

\begin{prop}\label{prop:right-seq}
If $\mu_\ell\subset k^\times$, there is a natural five-term exact sequence:
\begin{equation*} 
 \makeshort{ 0 \rto H^{q,\, q-1}(\X,\Z/\ell)
      \map{\zeta} K_q^M(k)/\ell \rto H^0(\X,\H^q(\muell{q}))
      \map{\partial} H^{q+1,q-1}(\X,\Z/\ell).}
  \end{equation*}
\end{prop}

\begin{proof}
Apply $H^q(\X,-)$ to the distinguished triangle in
Proposition~\ref{prop:coeff-triangle}. Using the fact that
$H^q(\X,C[j])=H^{q+j}(\X,C)$ 
and writing 
$\H^q$ for $\H^q(\muell{q})$, we get
\[
H^{-1}(\X,\H^q) \map{\partial} H^{q,q-1}(\X,\Z/\ell) \map{\zeta} 
H^{q,q}(\X,\Z/\ell) \rto H^0(\X,\H^q) \map{\partial} H^{q+1,q-1}(\X,\Z/\ell).
\]
The first term ($H^{-1}$) is $0$ because the coefficients are a sheaf.
By Lemma \ref{ovv2.2} with $p=q$, the third term is
$H^{q,q}(k,\Z/\ell)=K^M_q(k)/\ell$ \cite[Theorem 5.1]{mvw}.
\end{proof}

\begin{cor}\label{2->1:n=1}
Theorem \ref{thm:2->l} holds for $n=1$.
\end{cor}

\begin{proof}
By Proposition \ref{prop:5=>2}, we may assume $\zeta\in k$
so that $X=\Spec(E)$, $E=k(\root\ell\of a)$ and 
$X\times X=\coprod_GX$, where $G=\Gal(E/k)$.
By Lemma \ref{H0(Hq)}, $H^0(\X,\H^q)$ is
the equalizer of $H^q(X,\muell{q})\rrto \prod_G H^q(X,\muell{q})$,
i.e., $H^q(X,\muell{q})^G$.
Since $H^q(X,\muell{q})$ is $K^M_q(E)/\ell$, we have
$H^0(\X,\H^q) \cong (K^M_q(E)/\ell)^G$.
Proposition \ref{prop:right-seq} yields exactness of
\[
K_q^M(k)/\ell \rto (K^M_q(E)/\ell)^G \map{\partial} H^{q+1,q-1}(\X,\Z/\ell).
\]
Now combine this with the exact sequence \eqref{eq:n=1},
using Lemma \ref{lem:H90} to identify $\bar H_{-i,-i}(X)$.
\end{proof}

Our next goal, achieved in Corollary \ref{induced.map}, is
to connect the first map in Proposition \ref{prop:right-seq}
to the cup product with $\a$. We assume that $n\ge2$, so that 
$d=\dim(X)>0$ and $s_d(X)$ is defined.

\begin{prop}\label{prop:left-seq}
Let $X$ be a norm variety for $\a$ 
such that $s_d(X)\not\equiv0\pmod{\ell^2}$.
For $i\ge0$, there is a four-term exact sequence
\[\makeshort{
\bar{H}_{-i,-i}(X)_{(\ell)}\map{\pi_*} K_i^M(k)_{(\ell)} 
\map{r^*} H^{i+2d+1,i+d}(D,\Z_{(\ell)})\rto0.
 }\]
Suppose in addition that $X$ has a point of degree $\ell$.
Then the following sequence is exact:
\[\makeshort{
\bar{H}_{-i,-i}(X)\map{\pi_*} K_i^M(k)\map{r^*} H^{i+2d+1,i+d}(D,\Z_{(\ell)})\rto0.
 }\]
\end{prop}

\begin{proof}
Let $M$, $D$ and $\LL$ be as in Theorem \ref{thm:M}. 
Since $H^{p,q}(M[1]) = H^{p-1,q}(M)$, applying
$H^{i+2d+1,i+d}(-,\Z_{(\ell)})$ to the distinguished triangle in
\eqref{3.4.2} gives us the exact sequence
  \[\makeshort{
  H^{i+2d,i+d}(M,\Z_{(\ell)}) \map{} H^{i+2d,i+d}(\X\oo\LL^d) \map{r^*}
    H^{i+2d+1,i+d}(D,\Z_{(\ell)}) \map{u^*} H^{i+2d+1,i+d}(M,\Z_{(\ell)})
  }\]
where for brevity we have written $H^{p,q}(\X\oo\LL^d)$ for
$\Hom_{\DM}(\Rtr(\X)\oo\LL^d,\Z_{(\ell)}(q)[p])$. 
We will show that this may be rewritten as the 4-term sequence
of the proposition. 

The last term $H^{i+2d+1,i+d}(M,\Z_{(\ell)})$ vanishes because $M$ is
a direct summand of $X$, and $H^{p,q}(X,\Z_{(\ell)})=0$ whenever
$p-q>\dim(X)$; see \cite[3.6]{mvw}.  Similarly, the first term,
$H^{i+2d,i+d}(M,\Z_{(\ell)})$, is a summand of
$H^{i+2d,i+d}(X,\Z_{(\ell)})$, which we showed to be isomorphic to
$H_{-i,-i}(X,\Z_{(\ell)})$ if $i\ge0$, in the proof of Proposition
\ref{H-n-n}.  Therefore we may replace the first term by
$H_{-i,-i}(X,\Z_{(\ell)})$.
Since $X\!\rto\!\Spec(k)$ factors through $\X$, the map
$\pi_*:H_{-i,-i}(X,\Z_{(\ell)})\rto H_{-i,-i}(k,\Z_{(\ell)})=K^M_i(k)_{(\ell)}$
factors through $\bar H_{-i,-i}(X,\Z_{(\ell)})$, the
coequalizer of the two projections from $H_{-i,-i}(X\times X,\Z_{(\ell)})$.
%
We also know that
  \begin{align*}
H^{i+2d,i+d}(\X\oo\LL^d) &= \Hom_{\DM}(\X\oo\LL^d,\Z_{(\ell)}(i+d)[i+2d])
   = \Hom_{\DM}(\X, \Z_{(\ell)}(i)[i])  \\
   & = H^{i,i}(\X,\Z_{(\ell)}) \cong H^{i,i}(\Spec k,\Z_{(\ell)})
   \cong K_i^M(k)\oo\Z_{(\ell)} = K_i^M(k)_{(\ell)},
  \end{align*}
where the last two isomorphisms follow from Lemma \ref{ovv2.2}
and the Nestorenko-Suslin-Totaro Theorem \cite[5.1]{mvw}.
Thus we have constructed an exact sequence
\[\makeshort{
  \bar H_{-i,-i}(X,\Z_{(\ell)}) \map{\pi_*} K_i^M(k)_{(\ell)} \map{r^*}
  H^{i+2d+1,i+d}(D,\Z_{(\ell)}) \rto 0.
}\]

When $X$ has a point $x$ of degree $\ell$ over $k$,
every element $\alpha$ of $K^M_i(k)$ has $\ell\,\alpha=\pi_*([x,\alpha])$,
so the cokernel of $\pi_*:H_{-i,-i}(X)\rto H_{-i,-i}(k)=K^M_i(k)$ has
exponent $\ell$, and is the same as the cokernel of
$\bar{H}_{-i,-i}(X,\Z_{(\ell)}) \rto K_i^M(k)_{(\ell)}$. Thus we can replace
the first two terms of the exact sequence with these 
to get the desired sequence.
\end{proof}


\begin{cor}\label{induced.map}
If $\mu_\ell\subset k^\times$, there are maps 
$\alpha_i: H^{i+2d+1,i+d}(D,\Z_{(\ell)})\rto H^{n+i,n+i-1}(\grX,\Z/\ell)$
for all $i$ so that $\a\cup: K^M_{i}(k)/\ell\rto K^M_{n+i}(k)/\ell$
(the cup product with $\a$) factors as
\[
K^M_{i}(k)/\ell \rfib^{r^*} H^{i+2d+1,i+d}(D,\Z_{(\ell)})\map{\alpha_i} 
H^{n+i,n+i-1}(\grX,\Z/\ell) \rcofib^{\zeta} K^M_{n+i}(k)/\ell.
\]
\end{cor}

\begin{proof}
Set $q=n+i$. For each closed point $x$ of $X$, the diagram
\begin{diagram}
   { K^M_i(k(x))/\ell  & K^M_i(k)/\ell  \\
     K^M_q(k(x))/\ell  & K^M_q(k)/\ell  \\};
    \arrowsquare{N}{\a\cup=0}{\a\cup}{N}
\end{diagram}
commutes by the projection formula \cite[III.7.5.2]{WK}. Thus the map
$H_{-i.-i}(X)\rto K^M_q(k)/\ell$ is zero, since by
Proposition \ref{H-n-n} it is induced by the maps 
$$K^M_i(k(x))/\ell \map{N} K^M_i(k)/\ell \map{\a\cup}K^M_{q}(k)/\ell.$$
By Proposition \ref{prop:left-seq}, the cup product
factors through the quotient $H^{i+2d+1,i+d}(D,\Z_{(\ell)})$ of $K^M_i(k)/\ell$.
It remains to show that the image $\a K^M_i(k)$ of the cup product 
lands in the subgroup $H^{q,q-1}(\grX,\Z/\ell)$ of $K^M_q(k)/\ell$.
Since $H^0(X,\H^q(\muell{q}))$ is a subgroup of $K^M_{q}(k(X))/\ell$
(by Remark \ref{rem:11.1}),
it suffices by Proposition \ref{prop:right-seq} to show that 
$\a K^M_i(k)$ vanishes in $K^M_{q}(k(X))/\ell$. 
This is so because $k(X)$ splits $\a$.
\end{proof}




In Corollary \ref{cor:Q-surj},
we will show that the map $\alpha_i$ 
is an isomorphism. 
The inverse of $\alpha_i$ will be constructed 
using the cohomology operations $Q_i$ constructed in \cite[p.\,51]{RPO}.
Each $Q_i$ has bidegree $(2\ell^i-1,\ell^i-1)$;
see {\it loc.\,cit.} or \cite[13.3]{HW} for a summary of their properties. 
Thus the composite $Q=Q_{n-1}Q_{n-2}\cdots Q_0$ has bidegree 
$(2b\ell-n+2,b\ell-n+1)$, where $b=d/(\ell-1)$.

\begin{defn}
Define the $\Z$-graded ring $\HH^*(k)$
by
\[
\HH^i(-) = \bigoplus_{s\in \Z} H^{i+s,s}(-,\Z/\ell).
\]
In particular, $\HH^0(k)\cong K^M_*(k)/\ell$.
The cohomology operation $Q$ maps $\HH^i(Y)$ to $\HH^{i+b\ell+1}(Y)$.  Note
that $\HH^i(\tilde\X)=0$ for $i\le1$, by Lemma \ref{ovv2.2}.
\end{defn}

%


Now the operations $Q_j$ vanish on each $K^M_p(k)/\ell=H^{p,p}(k,\Z/\ell)$,
because $H^{p,q}(k,\Z/\ell)=0$ for $p>q$.
Since the $Q_j$ are derivations (\cite[13.10]{HW}), this means that $\HH^*(Y)$
is a graded $K^M_*(k)/\ell$-module for each $Y$, and each
$Q_j:\HH^i(Y)\map{}\HH^{i+\ell^j}(Y)$ is a 
$K^M_*(k)/\ell$-module homomorphism. Thus $Q:\HH^i \rto \HH^{i+b\ell+1}$
is also a $K^M_*(k)/\ell$-module homomorphism.


\begin{lem} \label{lem:Q-inj} 
Let $X$ be a norm variety over a field of characteristic~0, 
and let $\X$ be its 0-coskeleton.
Then the map
  $Q:\HH^1(\X)\rto\HH^{b\ell+2}(\X)$
is an injection.  
\end{lem}


\begin{proof}
Since $H^{p,q}(\Spec k,\Z/\ell) = 0$ for $p>q$, we have 
$\HH^i(\Spec k)=0$ for $i>0$. This yields isomorphisms
$\HH^i(\X)\map{\cong}\HH^{i+1}(\tilde\X)$ for all $i>0$.
In particular, $\HH^1(\X) \cong \HH^2(\tilde\X)$.  Thus it
suffices to show that $Q$ is injective on $\HH^2(\tilde\X)$.
Setting $a(j)={2+\frac{\ell^j-1}{\ell-1}}$, $Q_{j-1}\cdots Q_0$
maps $\HH^2(\tilde\X)$ to $\HH^{a(j)}(\tilde\X)$.
In particular it suffices to show that $Q_j$ is injective on
$\HH^{a(j)}(\tilde\X)$ for all $0\le j\le n-1$.
%
Because $X$ is a norm variety, we know from \cite[3.2]{mc/2}
(or \cite[10.14]{HW}) and \cite[13.20]{HW} that
the Margolis sequence is exact for each $Q_j$, $j<n$:
\[
\HH^{a(j)-\ell^j}(\tilde\X) \map{Q_j} \HH^{a(j)}(\tilde\X) \map{Q_j} 
\HH^{a(j)+\ell^j}(\tilde\X).
\]
By Lemma \ref{ovv2.2}, 
the left term is zero because
$a(j)-\ell^j\le1$. The result follows.
\end{proof}

Since $X$ is a splitting variety, 
$\a$ vanishes in $K^M_n(k(X))/\ell$. By Remark \ref{rem:11.1},
$\a$ vanishes in $H^0(X,\H^n(\muell{n}))$. It follows from
Proposition \ref{prop:right-seq} (or \cite[6.5]{mc/l}) 
that there is a unique element $\delta$ in $H^{n,n-1}(\grX,\Z/\ell)$
whose image in $K^M_n(k)/\ell$ is $\a$.

In the following proposition, $\zeta$ is the map defined in
Proposition \ref{prop:right-seq}, $\alpha$ is the 
direct sum of the maps $\alpha_i$ defined 
in Corollary \ref{induced.map}, and the maps $r^*$, $s^*$
are given in Theorem \ref{thm:M}. 

\goodbreak
\begin{prop} \label{rsQ square}
If $s_d(X)\not\equiv0\pmod{\ell^2}$, 
the following diagram commutes up to sign, and
the top composite is multiplication by $\a$.
  \begin{diagram}
    { \HH^0(\X) & \HH^1(\X) & K^M_*(k)/\ell \\
      \HH^{d+1}(D) & \HH^{b\ell+2}(\X) \\};
    \arrowsquare{\delta\cup}{r^*}{Q}{s^*} 
    \diagArrow{right hook->}{1-2}{1-3}^{\zeta} 
     \to{2-1}{1-2}^{\alpha}
  \end{diagram}
\end{prop}

\begin{proof}
Note that 
all maps in the diagram are (right) module maps
over the ring $K^M(k)/\ell\cong\HH^0(\X)$.
This is clear for multiplication by $\delta$, and we have already
seen that the cohomology operation $Q$ is also a $\HH^0(\X)$-module 
map. Finally, the maps $r^*$ and $s^*$ are also $\HH^0(\X)$-module maps, since
they come from morphisms in $\DM$; see \eqref{3.4.1} and \eqref{3.4.2}. 
 
The top row sends $x\in\HH^0(\X)$ to $\zeta(\delta\cup x)=\a\cup x$;
since $\zeta$ is an injection (by Proposition \ref{prop:right-seq}), 
and $\a\cup x=\zeta\circ\alpha^*r^*(x)$,
the upper triangle commutes: $\delta\cup x=\alpha^*r^*(x)$.

We will show that $s^*r^*(1) = (-1)^{n-1} Q(\delta)$. 
By linearity, it will follow that $s^*r^*(x)=(-1)^{n-1} Q(\delta\cup x)$
for all $x\in \HH^0(\X)$. Since $r^*$ is surjective by
Proposition \ref{prop:left-seq}, the result will follow.

We need to recall the definition of $\phi_V(\mu)$ from 
\cite[p.\,413]{mc/l} and \cite[5.10]{HW}.
Given an element $\mu$ in $H^{2b+1,b}(\grX,\Z/\ell)$,
form the triangle $A\rto\grX\map{\mu}\grX(b)[2b+1]$ and set 
$S=S^{\ell-2}A$. Then $\phi_V(\mu)$ is defined to be the element 
of $H^{2b\ell+2,b\ell}(\grX,\Z/\ell)$ represented by the composition
\[
\Rtr(\grX) \map{s} S(b)[2b+1] \map{r\oo1} \Rtr(\grX)(b\ell)[b\ell+2].
\]
When $\mu=Q_{n-2}\cdots Q_0(\delta)$, we get the distinguished
triangles \eqref{3.4.1} and \eqref{3.4.2} with $D=S$. Thus
the composition $s^*\circ r^*$ in the above diagram is multiplication 
by the element $\phi^V(\mu)$.
As observed in {\it loc.\,cit.}
By \cite[Thm.\,3.8]{mc/l} (cf.\,\cite[Cor.\,6.33]{HW}),
$\phi^V$ agrees with $\beta P^b$. 
In addition, since $\mu$ is annihilated by the $Q_i$ with $i\le n-2$ 
we have $\beta P^b(\mu)=(-1)^{n-1}Q_{n-1}(\mu)$; see
\cite[p.\,427]{mc/l} or \cite[5.14]{HW}.
This shows that the bottom right triangle commutes in the above diagram.
\end{proof}

\begin{remm}
In the proof of Proposition \ref{rsQ square}, we have cited
Definition 5.10, Corollary 6.33 and Lemma 5.14 from the book \cite{HW}.
These are slightly improved versions of Lemma 3.2 and (5.2), 
Theorem 3.8 and Lemma 5.13 in Voevodsky's paper \cite{mc/l}. 
Note that \cite[5.13]{mc/l} is missing several minus signs.
\end{remm}


\begin{cor} \label{cor:Q-surj}
In Proposition \ref{rsQ square},
$Q$ and $\alpha$ are isomorphisms, and the maps
$r^*$ and $\delta\cup-$ are surjections.
\end{cor}

\begin{proof}
From Proposition~\ref{prop:left-seq}, we see that $r^*$ is surjective.
By \cite[4.16]{HW}, $s^*$ is an isomorphism (because $d+1>d$),
and $Q$ is an injection by Lemma \ref{lem:Q-inj}.
The results follows from a diagram chase.
\end{proof}

Note that $H^{q,q-1}(\grX)=0$ for $q<n$, because by
Corollary \ref{cor:Q-surj} this is a quotient of $H^{q-n,q-n}(\grX)$.
Recall that $\widetilde{K}^M_q(k(x))/\ell$ is the equalizer of the two maps
\[
\iota_1, \iota_2: K^M_{q}(k(X))/\ell \rightrightarrows 
K^M_{q}(k(X\times X))/\ell.
\]

The following result was proved for $n=1$ in Corollary \ref{2->1:n=1},
and will be proved for $n\ge2$ in the next section.

\begin{prop} \label{prop:4thterm}
$H^0(\X, \H^q(\mu_\ell^{\otimes q})) \cong \tilde K^M_q(k(X))/\ell.$
\end{prop}

We are now ready to prove Theorem~\ref{thm:2->l} when $n\ge2$.

\begin{proof}[Proof of Theorem~\ref{thm:2->l}] 
Putting Proposition~\ref{prop:right-seq} for $q=n+i$ and
Proposition~\ref{prop:left-seq} together, we get that the
rows are exact in the following diagram, where $H^{p,q}(-)$
denotes $H^{p,q}(-,\Z/\ell)$.
\begin{diagram}
  { & \bar{H}_{-i,-i}(X) & K_i^M(k) & H^{i+2d+1,i+d}(D,\Z_{(\ell)}) \\
H^{q+1,q-1}(\X) & H^0(\X,\H^{q}(\muell{q})) 
&K_q^M(k)/\ell& H^{q,q-1}(\X) \\};
  \to{1-2}{1-3} \fib{1-3}{1-4}^{r^*}
  \diagArrow{left hook->}{2-4}{2-3}^{\zeta} \to{2-3}{2-2} \to{2-2}{2-1}
  \to{1-3}{2-3}^{\a\cup} 
  \diagArrow{densely dashed,->}{1-4}{2-4}^{\alpha}
  \to{1-3}{2-4}^{\delta\cup}
  \draw[densely dotted,->,rounded corners=5pt] 
       (m-1-2.south) -- (m-1-3.south west) -- (m-2-3.north west) 
       -- (m-2-2.north) -- (m-2-1.north east);
\end{diagram}
From Corollary \ref{cor:Q-surj} we can conclude that 
the five-term sequence indicated by the dotted arrow is exact:
\addtocounter{equation}{-1}
\begin{subequations}
\renewcommand{\theequation}{\theparentequation.\arabic{equation}}
\begin{equation}\label{eq:6-term}
  \bar{H}_{-i,-i}(\X) \rto K_i^M(k) \map{\a\cup} K_{q}^M(k)/\ell \rto
  H^0(\X,\H^{q}(\muell{q}))   
\rto H^{q+1,q-1}(\X).
\end{equation}
\end{subequations}
Theorem \ref{thm:2->l} now follows from Proposition~\ref{prop:4thterm}.
\end{proof}

\section{The fourth term}

Let $\iota_1,\iota_2$ be the two inclusions $k(X) \rcofib k(X\times X)$ induced
by the projections $X \times X \rto X$. 
To finish the proof of Theorem \ref{thm:2->l}, we need to 
prove Proposition \ref{prop:4thterm} for $n\ge2$.

\begin{lem}\label{lem:Gersten}
Fix $n\ge2$.  In the commutative diagram
\begin{diagram}
 {
      H^0(\X,\H^q) & H^0(X,\H^q) & H^0(X\times X, \H^q) \\
      E_0 & K^M_q(k(X))/\ell & K^M_q(k(X\times X))/\ell  \\
      E_1 & \bigoplus_{x\in X^{(1)}} K^M_{q-1}(k(x))/\ell & \bigoplus_{y\in
        (X\times X)^{(1)}} K^M_{q-1}(k(y))/\ell \\};
    \cofib{1-1}{1-2} \cofib{2-1}{2-2} \cofib{3-1.mid east}{3-2.mid west}
    \diagArrow{->,shiftup}{1-2}{1-3} \diagArrow{->,shiftdown}{1-2}{1-3}
    \diagArrow{->,shiftup}{2-2}{2-3}^{p_0} \diagArrow{->,shiftdown}{2-2}{2-3}_{p_0'}
    \diagArrow{->,shiftup}{3-2.mid east}{3-3.mid west}^{p_1}
    \diagArrow{->,shiftdown}{3-2.mid east}{3-3.mid west}_{p_1'}
    \cofib{1-1}{2-1} \cofib{1-2}{2-2} \cofib{1-3}{2-3}
    \to{2-1}{3-1} \to{2-2}{3-2} \to{2-3}{3-3}
  \end{diagram}
all of the columns are exact, 
and each $E_i$ is the equalizer of the morphisms $p_i$ and $p'_i$.
\end{lem}

\begin{proof}
Exactness of the first row (i.e., that $H^0(\X,\H^q)$ is the equalizer)
is immediate from Lemma \ref{H0(Hq)}.
The two right-hand columns are exact, as they are obtained from the Gersten
resolutions for $\H^q$.  The homomorphisms which are known to be injective are
denoted $\rcofib$.  By an elementary diagram chase, the left-hand column is also
exact.  
\end{proof}

In order to prove Proposition \ref{prop:4thterm} 
it thus suffices to show that $E_1 \cong 0$ in Lemma \ref{lem:Gersten}.

%
%
%

\begin{lem} If $n\ge2$, $E_1 = \ker p_1 = \ker p_1'.$
\end{lem}

\begin{proof}
Since $n>1$, we have $\dim X = \ell^{n-1}-1 \ge 1$.  
For any point $x\in X^{(1)}$ the summand indexed by $x$ is mapped
by $p_1$ and $p_1'$ to the summands indexed by the generic points of 
$x\times X$ and $X\times x$, respectively. Since these points (and
hence the summands) are distinct, the images of $p_1$ and $p_1'$
intersect in 0. It follows that their equalizer is
$\ker(p_1)=\ker(p_1')$, as asserted.
\end{proof}

\begin{prop}
If $X$ is a smooth variety of dimension $\ge1$,
then $p_1$ is injective.
\end{prop}

\begin{proof}
For each $x\in X^{(1)}$, let $y_x$ denote a generic point of 
$x\times X$; since $X$ is smooth, 
$x\times X$ is reduced. 
We will show that the composition of $p_1$ with the 
projection $\pi_x$ onto $K^M_{q-1}(k(y_x))/\ell$,
\[
\bigoplus_{x\in X^{(1)}} K^M_{q-1}(k(x))/\ell \map{p_1} 
\bigoplus_{y\in(X\times X)^{(1)}} K^M_{q-1}(k(y))/\ell \map{\pi_x}
K^M_{q-1}(k(y_x))/\ell,
\]
is an injection on the $x$-summand; since $\pi_x p_1$ is zero on all the
other summands of the left term, it will follow that $p_1$ is an injection. 

Fix $x$ and write $F$ for $k(X)$; as $X$ is smooth, 
the function field of $x\times X$ is 
a finite product of fields.
Choosing an affine neighborhood $\Spec R$ of $x$, $x$ is given
by a height~1 prime ideal $\frakp$ of $R$: $k(x)=\rmfrac(R/\frakp)$
and $F=\rmfrac(R)$. Note that $k(x)\oo R$ is a regular ring
because $X$ is smooth over $k$.
The kernel $\frakm$ of the multiplication map
\[
k(x)\oo R \rto k(x)\oo k(x) \map{\mu} k(x)
\]
is a maximal ideal of $k(x)\oo R$, and the localization $R'=(k(x)\oo
R)_{\frakm}$ at $\frakm$ is a regular local ring with residue field $k(x)$ and
fraction field $k(y_x)$.  Choose a regular sequence $r_1,\dots,r_d$ generating
the maximal ideal of $R'$; by iterated use of \cite[III.7.3]{WK}, there is a
{\it specialization map}
\[
K^M_*(k(y_x)) \map{\lambda} K^M_*(k(x))
\]
which is a left inverse to the component 
$p_1^x:K^M_*(k(x))\rto K^M_*(k(y_x))$ of $p_1$.
\end{proof}

Proposition \ref{prop:4thterm} now follows for $n\ge2$, 
since norm varieties are smooth by definition.  
This completes the proof of Theorem~\ref{thm:2->l}.

\bibliographystyle{alpha} 
\bibliography{WZ}

\begin{thebibliography}{MVW06}

\bibitem[BT]{BT}
H. Bass and J. Tate,
\emph{The Milnor ring of a global field}, pp. 349--446 in
\emph{Algebraic $K$-theory II},
Lecture Notes in Math. vol. 342, Springer, 1973.

\bibitem[B02]{Becher}
K. J. Becher,
\emph{Milnor $K$-groups and finite field extensions},
$K$-theory 27 (2002), 245--252.

\bibitem[D]{D}
P. Deligne,
Th\'eorie de Hodge III,
{\em Publ. Math. Inst. Hautes \'Et. Sci.} 44 (1974), 5--77.

\bibitem[FV]{biv}
E.~M. Friedlander and V.~Voevodsky, \emph{{Bivariant Cycle Cohomology}},
pp.~138--187 in \emph{Cycles, transfers, and motivic homology theories}, 
Annals of Mathematics Studies, vol. 143, Princeton University Press, 2000.

\bibitem[HW09]{HW-Rost}
C. Haesemeyer and C. Weibel.
\emph{Norm {V}arieties and the chain lemma (after {M}arkus {R}ost)}, 
Abel Symposium, Springer, 2009, pp.~95--130.

\bibitem[HW]{HW}
C. Haesemeyer and C. Weibel.
\newblock The norm residue theorem in motivic cohomology.
\newblock In progress.

\bibitem[Kel1y]{Shanekelly}
S. Kelly, 
\emph{Triangulated categories of motives in positive characteristic},
Preprint (Ph.D.\ thesis). Archived at arXiv:1305.5349, 2013.

\bibitem[M82]{M82}
A. Merkurjev,
\emph{Brauer groups of fields}, Comm. Alg. 11 (1983), 2611--2624.

\bibitem[MS]{MS}
A. Merkurjev and A. Suslin,
\emph{{$K$}-cohomology of {S}everi-{B}rauer varieties and the 
norm residue homomorphism}, Izv. Akad. Nauk SSSR Ser. Mat.
46 (1982), no.~5, 1011--1046, 1135--1136. 

\bibitem[MS2]{MS2} A. Merkurjev and A. Suslin, \emph{Motivic cohomology of the
    simplicial motive of a Rost variety}, J.  Pure and Appl. Algebra 214:2017--2026 (2010).

\bibitem[MVW]{mvw}
C. Mazza, V. Voevodsky, and C. Weibel,
\newblock {\em Lecture notes on motivic cohomology}, 
volume~2 of {\em Clay  Mathematics Monographs}.
\newblock American Mathematical Society, Providence, RI; Clay Mathematics
  Institute, Cambridge, MA, 2006.

\bibitem[OVV]{ovv}
D.~Orlov, A.~Vishik, and V.~Voevodsky.
\newblock An exact sequence for {$K^M_\ast/2$} 
with applications to quadratic forms.
\newblock {\em Ann. of Math. (2)}, 165(1):1--13, 2007.

\bibitem[SJ]{SJ}
A. Suslin and S. Joukhovitski, 
\emph{Norm varieties}, J. Pure Appl. Algebra \textbf{206} (2006), 
no.~1-2, 245--276. \MR{MR2220090 (2008a:14015)}


\bibitem[V/2]{mc/2}
V. Voevodsky.
\newblock Motivic cohomology with {${\bf Z}/2$}-coefficients.
\newblock {\em Publ. Math. Inst. Hautes \'Etudes Sci.}, 98:59--104, 2003.

\bibitem[V-ops]{RPO}
V. Voevodsky, 
\emph{Reduced power operations in motivic cohomology},
Publ. Math. Inst. Hautes \'Etudes Sci. (2003), no.~98, 1--57. \MR{MR2031199
  (2005b:14038a)}

\bibitem[V/l]{mc/l}
V. Voevodsky, 
\emph{On motivic cohomology with ${\bf Z}/l$-coefficients},
Annals of Math. 174 (2011), 401--438. \MR{MR2811603 (2012j:14030)}

\bibitem[V-CH]{V-CH}
V. Voevodsky, 
\emph{Motivic cohomology groups are isomorphic to higher Chow groups
in any characteristic}, Int. Math. Res. Notices 7 (2002), 351--355.

\bibitem[W90]{W-Top}
Charles Weibel, \emph{The norm residue isomorphism theorem}, J. Topology 2:346-372, (2009).

\bibitem[WH]{WH}
C. Weibel,
\newblock \emph{An introduction to homological algebra},
\newblock Cambridge Univ. Press, 1994.

\bibitem[WK]{WK}
Charles Weibel,
\newblock \emph{The $K$-book},
\newblock AMS Grad. Studies in Math. 145, 2013.

\bibitem[Y]{Yagita}
N. Yagita, 
\emph{Algebraic $BP$-theory and norm varieties}, Hokkaido Math J., 41:275-316,
2012. 

\end{thebibliography}

\end{document}